\documentclass[a4paper,11pt]{article}
\usepackage[english]{babel}
\usepackage[T1]{fontenc}
\usepackage{lmodern}
\usepackage[utf8]{inputenc}
\usepackage[all]{xy}
\usepackage{amsmath}
\usepackage{amsfonts}
\usepackage{amssymb}
\usepackage{enumerate}
\usepackage{amsthm}
\usepackage{multirow}
\usepackage{graphicx}
\usepackage{epsfig}
\usepackage{fancybox}
\usepackage{moreverb}
\usepackage{textcomp}
\usepackage{pifont}
\usepackage{float}

\usepackage{hyperref}
\usepackage[top=2cm,bottom=3cm,left=3cm,right=3cm]{geometry}

\title{Realizing  metrics of curvature $\leq -1$ on closed surfaces in Fuchsian anti-de Sitter manifolds}
 
\author{Hicham Labeni}

\begin{document}

\setcounter{secnumdepth}{2}
\setcounter{tocdepth}{3}

\newtheorem{definition}{Definition}[section]
\newtheorem{lemma}[definition]{Lemma}
\newtheorem{sublemma}[definition]{Sublemma}
\newtheorem{corollary}[definition]{Corollary}
\newtheorem{proposition}[definition]{Proposition}
\newtheorem{theorem}[definition]{Theorem}
\newtheorem{fact}[definition]{Fact}
\newtheorem{question}[definition]{Question}
\newtheorem{remark}[definition]{Remark}
\newtheorem{example}[definition]{Example}
\newtheorem{assumption}[definition]{Assumption}

\newcommand{\cov}{\mathrm{covol}}
\def \tr{{\mathrm{tr}}}
\def \det{{\mathrm{det}\;}}
\def\co{\colon\tanhinspace}
\def\I{{\mathcal I}}
\def\N{{\mathbb N}}
\def\R{{\mathbb R}}
\def\Z{{\mathbb Z}}
\def\Sph{{\mathbb S}}
\def\Tor{{\mathbb T}}
\def\Disk{{\mathbb D}}
\def\Hess{\mathrm{Hess}}
\def\rad{\mathbf{v}}

\def\H{{\mathbb H}}
\def\RP{{\mathbb R}{\mathrm{P}}}
\def\dS{{\mathrm d}{\mathbb{S}}}
\def\Isom{\mathrm{Isom}}

\def\pr{\mathrm{pr}}

\def\sh{\mathrm{sinh}\,}
\def\ch{\mathrm{cosh}\,}
\newcommand{\arccosh}{\mathop{\mathrm{arccosh}}\nolimits}
\newcommand{\oh}{\overline{h}}

\newcommand{\mf}{\mathfrak}
\newcommand{\mb}{\mathbb}
\newcommand{\ol}{\overline}
\newcommand{\la}{\langle}
\newcommand{\ra}{\rangle}
\newcommand{\hess}{\mathrm{Hess}\;}
\newcommand{\grad}{\mathrm{grad}}
\newcommand{\M}{\mathrm{MA}}
\newcommand{\II}{\textsc{I\hspace{-0.05 cm}I}}
\renewcommand{\d}{\mathrm{d}}
\newcommand{\A}{\mathrm{A}}
\renewcommand{\L}{\mathcal{L}}
\newcommand{\FB}[1]{{\color{red}#1}}
\newcommand{\note}[1]{{\color{blue}{\small #1}}}
\newcommand{\Area}{\mathrm{Area}}

%%%%%%%%%%%%%%%%%%%%%%%%%%%%%%%%%%%
%%%%%%%% gere espace texte formule
\setlength{\abovedisplayshortskip}{1pt}
\setlength{\belowdisplayshortskip}{3pt}
\setlength{\abovedisplayskip}{3pt}
\setlength{\belowdisplayskip}{3pt}

\newtheorem{exemple}{\textbf{Exemple}}[section]

\newtheorem{lemme}{\textbf{Lemma}}[section]

\newtheorem{corollaire}{\textbf{Corollaire}}[section]

\maketitle

\begin{abstract}We  prove that any metric with curvature $\leq -1$  (in the sense of A. D. Alexandrov) on a closed surface of genus $>1$ is isometric to the induced intrinsic metric on a space-like convex surface in a Lorentzian manifold of dimension $(2+1)$ with sectional curvature $-1$.
The proof is done by approximation, using a result about isometric immersion of smooth metrics by Labourie--Schlenker.
 \end{abstract}

%\tableofcontents

%\blfootnote{

%Hicham Labeni

%Département de mathématiques

%Université de Cergy-Pontoise, 95000 Cergy-Pontoise, France

%email: hicham.labeni@u-cergy.fr}

\section{Introduction} 

In the following, $S$ is a closed connected oriented surface. 
When we speak about a metric with curvature $\leq k$ or $\geq k$, this means that $S$ is endowed with a distance $d$ satisfying 
a curvature bound in the sense of A.D. Alexandrov, see e.g. \cite{BurBur} or Section~\ref{sec:approx smooth}.
This metric notion of curvature bound was initially introduced in the $40'$s to characterize the induced metric on the boundary of convex bodies of the Euclidean space \cite{AL06}. (In the present article, the word metric is used for distance, and \emph{induced metric} means the induced intrinsic distance.) While introducing this seminal notion, Alexandrov proved the following statement.

\begin{theorem}\label{1.1}
Let $d$ be a metric of curvature $\geq 0$ on the sphere $S$. Then there exists  a convex surface in the Euclidean space whose induced  metric is isometric to $(S,d)$.
\end{theorem}

Theorem~\ref{1.1} was generalized in many ways. Some of them are contained in the following statement, see the introduction of \cite{fgi}
for details.

 \begin{theorem}\label{gen riem}
Let $k\in \mathbb{R}$ and let $d$ be a metric of curvature $\geq k$ on a closed surface $S$. Then there exists a  Riemannian manifold $R$ homeomorphic to $S\times\mathbb{R}$ of constant sectional curvature $k$ which contains a convex surface whose induced  metric is isometric to $(S,d)$
\end{theorem}

In 2017, F. Fillastre and  D. Slutsky proved an analogous results for metrics with curvature bounded from above \cite{FFDS}. 
\begin{theorem}\label{1.2}
Let  $d$ be  a  metric  of   curvature $\leq 0$  on a closed surface $S$ of genus $>1$. Then there exists a flat   Lorentzian manifold  $L$ homeomorphic to $S\times\mathbb{R}$ which  contains  a  space-like  convex  surface whose  induced metric is isometric to $(S,d)$.
\end{theorem}

A natural question is to know if an analog of Theorem~\ref{gen riem} holds for metrics 
 with curvature bounded from above. The case $k=0$ is given by Theorem~\ref{1.2}. In the present article, we solve the case when $k$ is negative. Up to a homothety, this reduces to the case $=-1$. 
So the main result of the present paper is the following theorem. 
 
\begin{theorem}\label{1.4}
Let  $d$ be  a  metric  with  curvature $\leq -1$  on a closed surface $S$ of genus $>1$. Then there exists a Lorentzian manifold  $L$ of sectional curvature $-1$ homeomorphic to $S\times\mathbb{R}$ which  contains  a  space-like  convex  surface whose  induced metric is isometric to $(S,d)$.
\end{theorem}

The proof of Theorem~\ref{1.4} will be given by a classical approximation procedure, following the main lines of \cite{FFDS}. The proof relies on the smooth analogue of Theorem~\ref{1.4} proved by F. Labourie and J.-M. Schlenker, see Theorem~\ref{ScH}. We will prove Theorem~\ref{1.4} showing that the universal cover of $(S,d)$ is isometric 
to a convex surface in anti-de Sitter space (see Section~\ref{sec ads}), invariant under the action of a discrete group of isometries leaving invariant a totally geodesic hyperbolic surface. Such groups are usually called \emph{Fuchsian}, and the quotient of a suitable part of anti-de Sitter space by such a group may be called a Fuchsian anti-de Sitter manifold.  The main issues in our case, comparing to \cite{FFDS}, is that we lost the vector space structure given by the Minkowski space ---it is the Lorentzian analogue of the problem to go from Euclidean space to hyperbolic space. Also, the analogue of an approximation result that is straightforward in the flat case occupies the whole Section~\ref{sec:approx smooth} here.

Let us describe more precisely the content of the present article. In  Section~\ref{sec ads},  we  recall   some definitions related to anti-de Sitter space, and define the induced metric on convex surfaces in this space. 
In Section~\ref{sec:fuchsian} we look at surfaces invariant under the action of Fuchsian groups, and prove several compactness results.
  In section~\ref{sec:approx smooth}, we check that any  metric with curvature $\leq -1$ on $S$
 can be approximated  by a sequence of distances given by Riemannian metrics with sectional  curvature $<-1$.
 In Section~\ref{sec proof},  all the elements are put together to provide a proof of Theorem \ref{1.4}.

The case with a positive $k$ is still missing to obtain a Lorentzian analogue of Theorem~\ref{gen riem}. An issue is that it is not clear if the approximation results used in Section~\ref{sec:approx smooth} can be applied in the $\leq 1$ curvature  case.

\paragraph{Acknowledgement} This work is a part of author’s doctoral thesis completed under
the supervision of F. Fillastre
 and the author is invaluable grateful to him for his constant
attention and advice. The author is also invaluable grateful to an anonymous referee, whose comments 	undoubtedly helped
to greatly improve the presentation. 

\section{Convex surfaces in anti-de Sitter space}\label{sec ads}

\subsection{Anti-de Sitter space}

 In the following we will  describe  a geometric model of anti-de Sitter space (of dimension $3$)  we are most interested in, and illustrate some of its features. Good references for this material are \cite{Mess07}, \cite{BONSSCH}, \cite{BONSSCH10} and \cite{O'Neill83}.

 %Among the most important references for anti-de Sitter geometry, we cite the seminal  parper  of Geoffrey Mess  \cite{Mess07}, we cite also these good references,  \cite{BONSSCH}, and  \cite{O'Neill83}.
 
 % Good references for anti-de Sitter geometry are   \cite{Mess07},  \cite{BONSSCH},\cite{BONSSCH10}, and \cite{O'Neill83}.

% The main references for the material covered here are 

%We cite the seminal  parper  of Geoffrey Mess in 1990 \cite{Mess07}, we cite also theses good references,  \cite{BONSSCH},\cite{BONSSCH10}, \cite{O'Neill83},\cite{FS19} and \cite{FF11}.

 Let us consider the symmetric  bilinear form 
    \begin{center}
$b(x,y)=-x_0y_0-x_1y_1+x_2y_2+x_3y_3.$
\end{center}
of signature $(2,2)$ on    $\mathbb{R}^4$.

\begin{definition}\label{quadric}  
We define   $\widehat{AdS^3}$ as 
\begin{center}
$\widehat{AdS^3}=\{(x_0,x_1,x_2,x_3)\in\mathbb{R}^4|b(x,x)=-1\},$
\end{center}
endowed with the pseudo-Riemannian metric induced by  the restriction of the bilinear form $b$ to its tangent spaces.
\end{definition}

Hence $\widehat{AdS^3}$ is a Lorentzian manifold, and it can be checked  that its sectional curvature is $-1$.

$$
\text{A tangent vector $v$ to $\widehat{AdS^3}$ at a point  $x$ is called:  }%v\in T_x\widehat{AdS^3} \text{     }is\text{  }  
\left \{
    \begin{array}{ll}
         \text{space-like}& if \text{     }  b(v,v)>0. \\ 
        \text{time-like}& if \text{     } b(v,v)<0. \\
        \text{light-like}& if  \text{     } b(v,v)=0. 
    \end{array}
\right.
$$

Now let $x,y\in \mathbb{R}^4$.  We say that $x\sim y$ if and only if  there exists  $\lambda\in\mathbb{R}^*$  such that  $x=\lambda y$. % and we give the following definition. 
\begin{definition}  We define   the anti-de Sitter space  of dimension $3$ as follows: $$AdS^3=\widehat{AdS^3}/\sim$$ %$$AdS^3=\{(x_1,x_2,x_3,x_4)\in\mathbb{R}^4:-x_1^2-x_2^2+x_3^2+x_4^2<0\}/\sim\subset\mathbb{RP}^3,$$ 
endowed with the quotient metric.
\end{definition}
It is easy to see that $\widehat{AdS^3}$ is a double cover of $AdS^3$. The pseudo-Riemannian metric induced on $\widehat{AdS^3}$ goes down to the quotient. 

By definition $AdS^3$ is a subset of the projective space. 
In order to better visualize  it, we look at its intersection with an affine chart and see its image in $\mathbb{R}^3$. Let   $\varphi_0:\mathbb{RP}^3\setminus\{x_0=0\}\rightarrow\mathbb{R}^3$ be an affine chart  of $\mathbb{RP}^3$ defined by:
\begin{equation}
\varphi_0([x_0,x_1,x_2,x_3])=\left(\frac{x_1}{x_0},\frac{x_2}{x_0},\frac{x_3}{x_0}\right)=(\bar{x}_1,\bar{x}_2,\bar{x}_3). \label{cord}
\end{equation}

Then  $\varphi_0(AdS^3\setminus\{x_0=0\})$ gives,

$$-x_0^2-x_1^2+x_2^2+x_3^2<0 \Rightarrow-1-\left({\frac{x_1}{x_0}}\right)^2+\left({\frac{x_2}{x_0}}\right)^2+\left({\frac{x_3}{x_0}}\right)^2<0$$ so in this  affine chart $AdS^3$ fills the domain $$-\bar{x}_1^2+\bar{x}_2^2+\bar{x}_3^2<1,$$ which is the interior of a one-sheeted hyperboloid. Notice that $AdS^3$ is not contained in a single affine chart. In the affine  chart  $\varphi_0$  we are missing a totally geodesic plane at infinity, corresponding to $\{x_0=0\}$.

In all the article, we will denote by $\mathbb{D}$, the disc  
$ 
\left \{
    \begin{array}{ll}
     \bar{x}_2^2+\bar{x}_3^2  <1 \\ 
\bar{x}_1=0 \\     \end{array}
\right.
$ in the affine chart $\varphi_0$ (see Figure \ref{adsfig4}).

%\begin{center}
\begin{figure}[H]
  \includegraphics[scale=0.15]{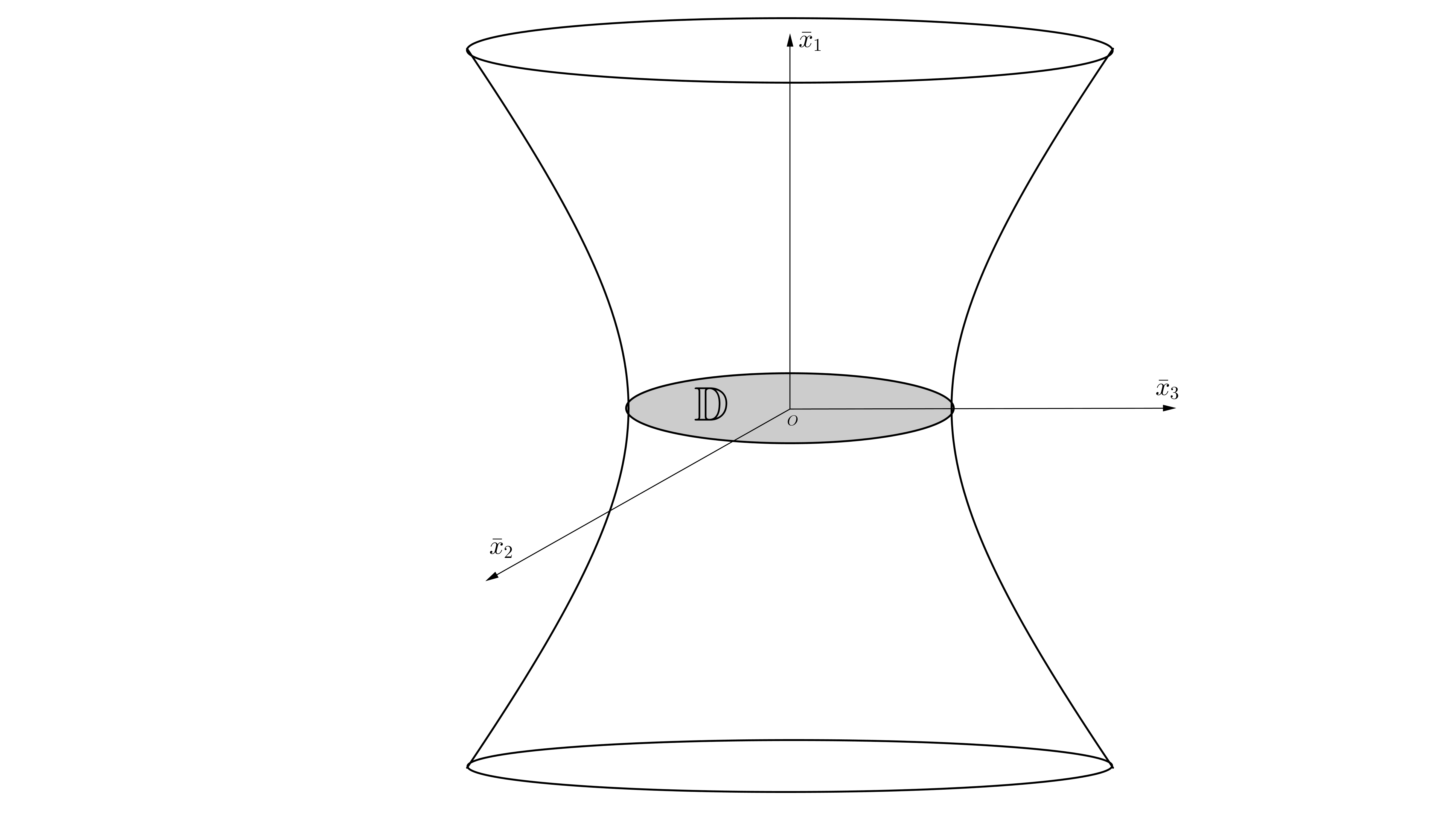}
 \begin{center}  

   \caption{Image of $AdS^3$ in the affine chart $\varphi_0$.}\label{adsfig4}
   \end{center}
\end{figure}
%\end{center}

It is clear from the construction that in the affine chart $\varphi_0$, geodesics (resp. totally geodesic planes) are given by the intersection between affine lines (resp. affine planes) in $\mathbb{R}^3$ with the interior of the one sheeted hyperboloid described above.  
 A plane $P$ is space-like if the restriction of the induced metric on $P$ is positive-definite. A \emph{convex  space-like surface}  in anti-de Sitter space is a surface which is  convex in an affine chart and   which has only space-like planes as support planes. The \emph{boundary at infinity} of   $AdS^3$  is given by \begin{center}
 $\{[x]\in\mathbb{RP}^3:b(x,x)=0\}/\sim$
 \end{center}
 and we will denote it by $\partial_\infty AdS^3$,  (and by  $\varphi_0(\partial_\infty AdS^3)$ the boundary in the affine chart  $\varphi_0$).
 We can distinguish the type of geodesics in the image of anti-de Sitter space in the affine chart as follows (see  Figure~\ref{geodesicss}): \begin{itemize}
 \item A geodesic in  $AdS^3$ is space-like if it  meets  $\partial_\infty AdS^3$ in two different points. 
  \item A geodesic in $AdS^3$ is light-like if it meets   $\partial_\infty AdS^3$ in only one point. 
    \item A geodesic in $AdS^3$ is time-like if it is strictly contained  in the hyperboloid. 
  \end{itemize}
   \begin{center}
 \begin{figure}[H]
  \includegraphics[scale=0.15]{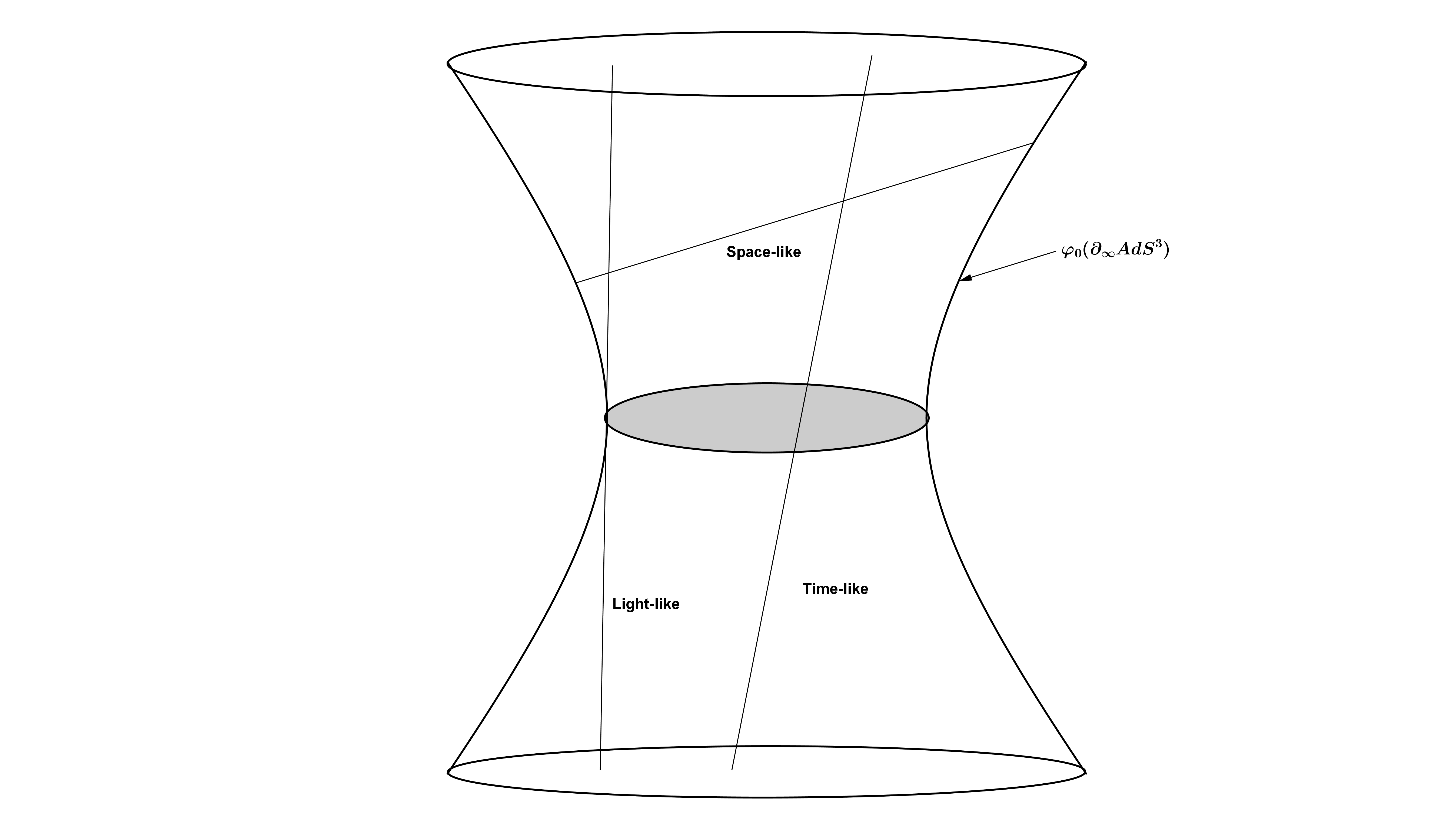}
  \caption{Geodesics in an affine model of $AdS^3$.}\label{geodesicss}
 \end{figure}
  \end{center}

  Note  that $AdS^3\cap \{x\in \R^4 \vert x_1=0 \}=:H_0$ is isometric to the hyperbolic plane. We use this fact to define the following map.
Let $\tilde{\Psi}:\mathbb{H}^2\times\mathbb{R}\longrightarrow AdS^3$ be the map  defined by $\tilde{\Psi}(x,t)=\exp_x(tV)$ where 
\begin{itemize}
\item  $\tilde{\Psi}(\mathbb{H}^2,0)=H_0$, and $x\mapsto \tilde{\Psi}(x,0)$ is an isometry,
\item $V$ is a choice of a unit vector field orthogonal to $H_0$, for the anti-de Sitter metric.
\end{itemize}

Indeed, we have $\tilde{\Psi}(x,t)=\cos(t)x+\sin(t)V$ with $V=(0,-1,0,0)$.
For a given $x$, $t\mapsto  \tilde{\Psi}(x,t)$ is a time-like geodesic loop with time-length $2\pi$.
We will call   \emph{AdS cylinder} the cylinder $\mathbb{H}^2\times [0,\pi/2[$ endowed with the Lorentzian 
metric $g_{AdS}$, which is the pull back of the anti-de Sitter metric by $\tilde{\Psi}$. 
Let us denote  $AdS^3\cap \{x\in \R^4 \vert x_1=r \}=:H_r$. The induced metric onto $H_r$ is homothetic to the hyperbolic metric with  factor $(1-r^2)$, and clearly $\tilde{\Psi}(\mathbb{H}^2,t)=H_{\sin(t)}$. In turn, 
$$g_{AdS}(x,t)=\cos^2(t)g_{\mathbb{H}^2}(x)-dt^2$$
where $g_{\mathbb{H}^2}$ is the metric on the hyperbolic plane.

It will be suitable to work with the image of $\tilde{\Psi}$ in the affine chart considered above. Let us denote $\Psi=\varphi_0\circ \tilde{\Psi}$. The set $\Psi(\mathbb{H}^2\times [0,\pi/2[)$ 
 is indeed a Euclidean half-cylinder in $\mathbb{R}^3$ (see Figure~\ref{AdS-cylinder11}). We have $\Psi(\mathbb{H}^2,0)=\mathbb{D}$ and for $x\in \mathbb{H}^2$,
$t\mapsto  \tilde{\Psi}(x,t)$ is a vertical half line from $\mathbb{D}$. We will call \emph{affine AdS cylinder} the image of $\mathbb{H}^2\times[0,\pi/2[$ by $\Psi$. For convexity reasons, we will need only to consider a half cylinder.
   \begin{center}
\begin{figure}[H]
   \includegraphics[scale=0.15]{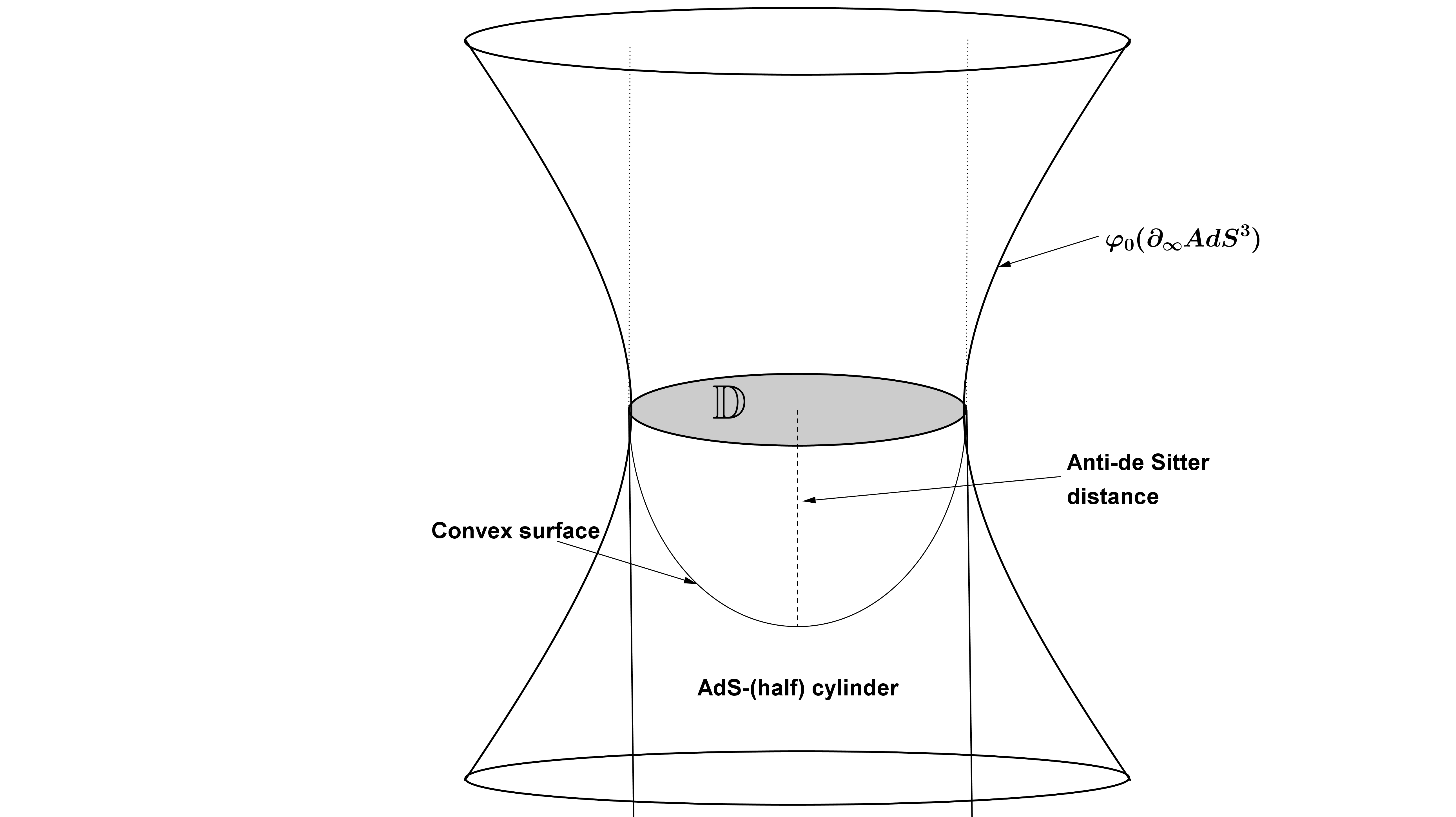}
 \caption{The AdS cylinder and a  convex surface inside.}\label{AdS-cylinder11}
 \end{figure}
   \end{center}

\subsection{Convex functions}

For a function   $u:\mathbb{H}^2\to[0,\pi/2[$, we denote 
$$S_u=\{(x,u(x))\vert x\in\mathbb{H}^2\}~.$$
For every  $x\in\mathbb{H}^2$ we denote by  $\bar{x}=\Psi(x,0)$ the  corresponding point on the disc $\mathbb{D}$, where $ \Psi$ it the map  introduced in the previous section.
The image of $S_u$ in the affine AdS cylinder  is the graph of a function over $\mathbb{D}$, that we will denote by 
 $\bar{u}$. We will denote by $S_{\bar{u}}$ the image of $S_u$.  Hence, $\bar{u}:\mathbb{D}\to\mathbb{R}$ and 
 $$(\bar{x},\bar{u}(\bar{x}))=\Psi(x,u(x))~.$$
 
For a point    $\bar{x}\in\mathbb{D}$ we use the notation  $\bar{x}=(\bar{x}_2,\bar{x}_3)$  for its  Euclidean coordinates, and its Euclidean norm is $\parallel\bar{x}\parallel=\sqrt{\bar{x}_2^2+\bar{x}_3^2}$. By the considerations of the preceding section, we immediately obtain the following relation.

\begin{lemma} \label{relation distances} With the notations above
$\bar{u}(\bar{x})=-\tan(u(x))\sqrt{1-\parallel \bar{x}\parallel^2}.$ 
\end{lemma}

\begin{definition}
Let $u:\mathbb{H}^2\rightarrow \mathbb{R}$ be a function. We say that $u$ is \emph{C-convex} if 
\begin{itemize}
\item $u\geq 0$ and there is $R<\pi/2$ such that $u\leq R < \pi/2$;
\item the corresponding function $\bar{u}$ is convex. 
\end{itemize}
 \end{definition}

  It is worth noting that for $R\geq 0$, if $u=R$, then  the graph of the map defined by
 $\bar{u}(\bar{x})=-\tan(R)\sqrt{1-\parallel \bar{x}\parallel^2}$  is a half ellipsoid. Also,
$|\bar{u}(\bar{x})|\leq\tan(R)\sqrt{1-\parallel \bar{x}\parallel^2}$.
 It follows that if $u$ is C-convex, then  $\bar{u}$ is bounded and  satisfies $\bar{u}|_{\partial\mathbb{D}}=0$.

  It is also clear   that   a bounded convex function $\bar{u}:\mathbb{D}\to \mathbb{R}$ vanishes everywhere if it vanishes in a point of the open disc $\mathbb{D}$.  So we have $\bar{u}\leq 0$ by definition, and $\bar{u}<0$ or $\bar{u}=0$.

Let us note the following.

\begin{lemma} \label{passe par le disc} In the image of $AdS^3$ by $\varphi_0$,
\begin{enumerate}
\item Every time-like  line passes through the disc $\mathbb{D}$. 
\item Every light-like  line which doesn't pass through the boundary $\partial_\infty\mathbb{D}$ must pass through the disc $\mathbb{D}$.
\item A cone with basis the disc $\mathbb{D}$ and with apex in the affine cylinder is a convex  space-like surface. 
\end{enumerate}
\end{lemma}
\begin{proof} 
The proofs of the two first points are almost immediate. For the third point, either a support plane of the cone does not meet 
the closure of $\mathbb{D}$, hence it is space-like, or  a support plane of the cone contains a half-line of the cone, then it meets the boundary of the disc, but by assumption this half-line is not vertical, hence not light-like, so the plane is space-like.
\end{proof} 
\begin{lemma}\label{convex then spacelike}
Let $u:\mathbb{H}^2\to[0,\pi/2[ $ be a C-convex  function.  Then the surface $S_u$ is space-like. 
\end{lemma}
\begin{proof} 
Let $p$ be a point on the image of $S_u$ in the affine half  cylinder, and let $C_p$ be the cone with basis the disc $\mathbb{D}$ and apex $p$. By definition, this cone is contained in the affine half  cylinder. By convexity, a support plane to the surface at $p$ is a support plane of the cone, so by  
 Lemma  \ref{passe par le disc} it must be space-like.
\end{proof}

We say that a sequence $(u_n)_n$ of C-convex functions is \emph{uniformly bounded} if there is $R<\pi/2$
such that for any $n$, $u_n<R$.
 \begin{lemma}\label{lem:subseq cv}
 Let $(u_n)_n$ be a sequence of uniformly bounded C-convex functions. Up to extracting  a subsequence, $(u_n)_n$ converges to a C-convex function $u$, uniformly on  compact sets. 
 \end{lemma}
 \begin{proof}
 This is a classical  property of the corresponding convex functions $\bar{u}_n$,   \cite[Theorem 10.9]{Roc97}, in the special case when the surfaces vanish on the boundary of the disc $\mathbb{D}$.
 \end{proof}

 Let $u_n$, $n>1$, be uniformly bounded C-convex  functions converging to a C-convex function $u=u_0$. 
 Let  $c:I\rightarrow\mathbb{H}^2$ be a Lipschitz curve and  $\bar{c}:I\rightarrow\mathbb{D}$ be its image by $\Psi$. 
  Then $\bar{u}\circ \bar{c}$, $\bar{u}_n\circ \bar{c}$ are Lipschitz ---the Lipschitz nature of $\bar{c}$ is independent of a choice of a Riemannian metric on the disc. By Rademacher  Theorem,  there exists a set $I_0$ of Lebesgue measure $0$ in $I$ such that for all $n\in\mathbb{N}$,  $\bar{u}_n$  is  differentiable  on $I\setminus I_0$ .

 \begin{lemma}\label{deriva} 
Let $u_n:\mathbb{H}^2\to \mathbb{R}$ be uniformly bounded  C-convex functions converging to a C-convex function $u$, and let $c:I\rightarrow\mathbb{H}^2$ be a Lipschitz curve.  Up to extracting a subsequence, for almost all $t$, %$u_n\circ c$ and $u$ are differentiable  and at such a point $t$, 
 $$(u_n\circ c)'(t)\rightarrow (u\circ c)'(t)~.$$ 
 \end{lemma}  
 \begin{proof}
 The following proof is a straightforward  adaptation of \cite[Lemma 3.6]{FFDS}. We first  prove the Lemma   for the corresponding functions $\bar{u}_n$ and $\bar{u}$, then we deduce   the proof for $u_n$ and $u$  using continuity and Lemma~\ref{relation distances}. We consider that $\bar{c}$ is parameterized by arc-length.
 
Let   $\langle \cdot ,\cdot\rangle$  be the  the Euclidean metric on the affine cylinder, and we use the notation $\begin{pmatrix}
 {a} \\
 {b} \\
 \end{pmatrix}$, with  $a\in\mathbb{D}$ and $b\in\mathbb{R}$. Let $t$ be such that the derivatives exist.
Let $X$ be the unit  vector  $
 \begin{pmatrix}
 {0} \\
 {1} \\
 \end{pmatrix}
 $
 and  $Y$ the unit vector $ \begin{pmatrix}
  {\bar{c}'(t)} \\
  {0} \\
  \end{pmatrix}
$, we have $ \langle X,Y \rangle $=0. The tangent vector to the curve $ \begin{pmatrix}
  {\bar{c}} \\
  {\bar{u}_n\circ \bar{c}} \\
  \end{pmatrix}
$     at every point  $ \begin{pmatrix}
  {\bar{c}(t)} \\
  {(\bar{u}_n\circ \bar{c})(t)} \\
  \end{pmatrix}
$ 
 is given by  $$V_n=(\bar{u}_n\circ \bar{c}(t))'X+Y$$ and in the plane $P$ spanned by $X$ and $Y$, the vector 
 $$N_n=(\bar{u}_n\circ \bar{c}(t))'Y-X$$ 
 is orthogonal to $V_n$ for $\langle\cdot,\cdot\rangle$. Now because  $\bar{u}_n$ and $\bar{u}$ are equi-Lipschitz on any compact set of $\mathbb{D}$   (see  \cite[Theorem 10.6]{Roc97})  %and because of the fact that $\bar{u}_n$ are uniformly bounded,   
 then  there exists $k$ such that $|(\bar{u}_n\circ \bar{c})'(t)|\leq k$ 
 for all $n\in\mathbb{N}$, then  
  $$\parallel N_n\parallel\leq |(\bar{u}_n\circ \bar{c})'(t)|\parallel Y\parallel+\parallel X\parallel \leq |(\bar{u}_n\circ \bar{c})'(t)|+1\leq k+1 $$
 so $\parallel N_n\parallel$ are    uniformly bounded. Hence, up to extracting a subsequence $(N_n)_n$ converges to a vector $N$. Note that $N$ is not the zero vector, otherwise $\langle N_n,X\rangle$ would converge to $0$, that is impossible because $\langle N_n,X\rangle=-1$.
 
 Let $T_n$ be the intersection of the convex surface $S_{\bar{u}_n}$ defined  by $\bar{u}_n$ and the plane $P$.  The set $T_n$ is a convex set in $P$, and $V_n$ is a tangent vector, hence by convexity for any $\bar{y}\in\mathbb{D}\cap P$,
 $$\langle N_n, \begin{pmatrix}
   {\bar{c}(t)} \\
   {\bar{u}_n\circ \bar{c}(t)} \\
   \end{pmatrix} - \begin{pmatrix}
     {\bar{y}} \\
     {\bar{u}_n(\bar{y})} \\
     \end{pmatrix}\rangle\geq 0,$$
and passing to the limit we get 
$$\langle N, \begin{pmatrix}
   {\bar{c}(t)} \\
   {\bar{u}\circ \bar{c}(t)} \\
   \end{pmatrix} - \begin{pmatrix}
     {\bar{y}} \\
     {\bar{u}(\bar{y})} \\
     \end{pmatrix}\rangle\geq 0,$$
     this says that $N$ is a normal vector to $T$ (the intersection of $S_{\bar{u}}$ with $P$), hence  
     $$\langle N, (\bar{u}\circ \bar{c})'(t)\begin{pmatrix}
        {0} \\
        {1} \\
        \end{pmatrix} + \begin{pmatrix}
          {\bar{c}'(t)} \\
          {0} \\
          \end{pmatrix}\rangle=0.$$
          So there exists $\lambda$ such that 
          $$ (\bar{u}\circ \bar{c})'(t)\begin{pmatrix}
                  {0} \\
                  {1} \\
                  \end{pmatrix} + \begin{pmatrix}
                    {\bar{c}'(t)} \\
                    {0} \\
                    \end{pmatrix} =\lambda\underset{n\rightarrow\infty}{\lim}(\bar{u}_n\circ \bar{c})'(t)\begin{pmatrix}
                            {0} \\
                            {1} \\
                            \end{pmatrix} + \begin{pmatrix}
                              {\lambda \bar{c}'(t)} \\
                              {0} \\
                              \end{pmatrix}.$$

 By identification it follows  that $\lambda=1$, hence $(\bar{u}_n\circ \bar{c})'(t)$ must converge to $(\bar{u}\circ \bar{c})'(t)$.
  The functions $\bar{u}_n\circ \bar{c}$ and $u_n\circ c$ are defined from $I\subset \mathbb{R}$ to $\mathbb{R}$, by Lemma~ \ref{relation distances}

$$u_n\circ c(t)=\arctan\left(\frac{\bar{u}_n\circ \bar{c}(t)}{h(t)}\right)$$

   %$$\bar{u}\circ c'(t)=.........\tan(u\circ c(t))\sqrt{1-c_1^2(t)-c_2^2(t)}$$
   
 where $h(t)=-\sqrt{1-\parallel \bar{c}(t)\parallel^2}$, hence  $u_n\circ c$ is clearly differentiable almost everywhere  for all $n\in\mathbb{N}$ and   
  % $$(\bar{u}_n\circ \bar{c})'(t)=(1+\tan^2(u_n\circ c(t)))(u_n\circ c)'(t)h(t)+\tan(u_n\circ c(t))h'(t)$$  

    \begin{equation}
     (u_n\circ c)'(t)=\frac{(\bar{u}_n\circ\bar{c})'(t)h(t)-(\bar{u}_n\circ \bar{c})(t)h'(t)}{h^2(t)+(\bar{u}_n\circ\bar{c})^2(t)}    \label{relderiv}
   \end{equation}

    also we have (by hypothesis) for almost all $t$, that 
    
    $$(u_n\circ c)(t)\underset{n\rightarrow\infty}{\longrightarrow} (u\circ c)(t)$$

       hence by  continuity  (in the relation given by Lemma \ref{relation distances}) it is clear that, 
    $$(\bar{u}_n\circ\bar{c})(t)\underset{n\rightarrow\infty}{\longrightarrow} (\bar{u}\circ\bar{c})(t)$$
    
    then by the preceding arguments and  by continuity again in \eqref{relderiv} and passing to the limit, it follows that $(u_n\circ c)'(t) $ converge to  $(u\circ c)'(t)$.
   \end{proof}

%we have $$\bar{u}_n\circ \bar{c}(t)=\tan(u_n\circ c(t))h(t)$$
  %because $(u_n\circ c)(t)$ converge to $(u\circ c)(t)$  then   by continuity we have 
%$$(\bar{u}_n\circ \bar{c})(t)\rightarrow (\bar{u}\circ \bar{c})(t)$$ for almost all $t$ ............ A VERIFIER C FACILE

Let $u:\mathbb{H}^2\to \mathbb{R}$ be a C-convex function. 
 For $c:[0,1]\to \mathbb{H}^2$ a Lipschitz curve, 
 $(c,u \circ c)$ is a curve on $S_u$, and its length for the anti-de Sitter metric is 
 \begin{equation}
\mathcal{L}_u (c)=\displaystyle{\int_{0}^{1}\sqrt{\cos^2(u\circ c(t))\| c'(t)\|_{\mathbb{H}^2}^2-(u\circ {c})'^2(t)}}dt~.       \label{lengthstructure}
 \end{equation}

  By  Lemma \ref{deriva} above and using the dominated convergence Theorem, we get the following proposition.
 
  \begin{proposition}\label{convstruc}
Let $u_n:\mathbb{H}^2\to \mathbb{R}$ be uniformly bounded  C-convex functions converging to a C-convex function $u$, and let $c:I\rightarrow\mathbb{H}^2$ be a Lipschitz curve. Up to extracting a subsequence, $\mathcal{L}_{u_n}(c)\rightarrow \mathcal{L}_u(c).$ 
   \end{proposition} 
  
 The induced (intrinsic) metric $d_{S_u}$ on $S_u$ is the  pseudo-distance    induced by $\mathcal{L}_u$: 
for $x,y\in S_u$, $d_{S_u}(x,y)$  is the  infimum of the lengths of Lipschitz curves between $x$ and $y$ contained in $S_u$. Note that as the AdS cylinder has a Lorentzian metric, the induced distance between two distinct points on $S_u$ may be equal to $0$, that is a major difference with the case of induced metrics on surfaces in a Riemannian space.

 \begin{definition} \label{d_u}
We denote  by $d_u$ the pull-back of $d_{S_u}$ on $\mathbb{H}^2$, so that for every point $x,y\in\mathbb{H}^2$
$$d_u(x,y)=d_{S_u}((x,u(x)),(y,u(y))).$$        
 \end{definition}

 From \eqref{lengthstructure}, as $\cos\leq 1$, we clearly have the following.
  
 \begin{lemma} \label{major}
 With the notations above, for $x,y\in \mathbb{H}^2$,
 $d_u(x,y)\leq d_{\mathbb{H}^2}(x,y).$
 \end{lemma}

\section{Fuchsian invariance}\label{sec:fuchsian}

\subsection{Convergence of surfaces implies convergence of metrics}
The aim of this section is to state Proposition~\ref{5.10}. The arguments are quite general and close to the ones of \cite{FFDS}. The main point is Lemma~\ref{minor} below, that is the AdS analogue of Corollary 3.11 in 
\cite{FFDS}.

 Recall that a Fuchsian group is a discrete group of orientation preserving isometries acting  on the hyperbolic plane.
 In the present article, we will restrict this definition to the groups acting moreover freely and  cocompactly. 
 \begin{definition} 
  A \emph{Fuchsian C-convex   function} is a couple $(u,\Gamma)$, where  $u$ is a C-convex function and $\Gamma$ is a Fuchsian  group such that for all $\sigma\in\Gamma$ we have 
  $u\circ\sigma=u$.
  \end{definition}
  
  We will often abuse terminology, speaking about Fuchsian for a single function $u$, so that the Fuchsian group will remain implicit.
  
\begin{definition}\label{3.9}
Let $(\Gamma_n)_n$ be a sequence of discrete groups.  $(\Gamma_n)_n$ converges to a group $\Gamma$  if  there exist isomorphisms $\tau_n:\Gamma\rightarrow\Gamma_n$ such that for all $\sigma\in\Gamma,\tau_n(\sigma)$ converge to $\sigma$. 
\end{definition} 
 
 \begin{definition}
We say that a sequence of Fuchsian C-convex functions $(u_n,\Gamma_n)_n$ converges to a pair $(u,\Gamma)$, if $u$ is a C-convex function,  $\Gamma$ is a Fuchsian group such that  $(u_n)_n$ converges to $u$ and $(\Gamma_n)_n$ converges to $\Gamma$.
 \end{definition} 
 
It is easy to see that if $(u_n,\Gamma_n)$ is  a  sequence  of  Fuchsian  C-convex  functions  that converges to a pair  $(u,\Gamma)$, then $(u,\Gamma)$ is a Fuchsian C-convex function, see e.g. \cite[Lemma 3.17]{FFDS}.
Recall the definition of the distance $d_u$ from Definition~\ref{d_u}. Recall also that a C-convex function is differentiable almost everywhere. At a point where $u$ is differentiable, we denote by $\parallel \cdot \parallel_u$ the norm induced 
by the ambient anti-de Sitter metric on the tangent of $S_u$ at this point.

\begin{lemma}\label{minor}\label{converg}
Let $u$ be a C-convex function. Let  $K:=\inf(\parallel v \parallel_u/\parallel v\parallel_{\mathbb{H}^2})$, and let $d_{\mathbb{H}^2}$ be the distance given by the hyperbolic metric (for instance, $d_{\mathbb{H}^2}=d_u$ for $u=0$). Then $d_u(x,y)\geq  K d_{\mathbb{H}^2}(x,y)$.

Moreover, if $u$ is Fuchsian, then $K>0$.
\end{lemma}
\begin{proof}
Let $c$ be a Lipschitz curve between two points $x,y \in \mathbb{H}^2$. Let 
$v$ be the tangent vector field of $(c,u\circ c)$ whenever it exists.
We have 
$$\mathcal{L}_u(c)=\int_a^b \parallel v\parallel_u \geq K\int_a^b \parallel v\parallel_{\mathbb{H}^2} \geq Kd_{\mathbb{H}^2}(x,y)$$
and the first result follows as by definition $d_u(x,y)$ is an infimum of lengths.

Now let us suppose that $u$ is Fuchsian. 
Let us suppose that $K=0$, i.e. there is a sequence $(x_n)_n$ such that $u$ is differentiable at each $x_n$, and 
 $v_n\neq 0$ in  $T_{x_n}\mathbb{H}^2$ such that $\parallel v_n \parallel_u/\parallel v_n\parallel_{\mathbb{H}^2}\to 0$. Without loss of generality, let us consider that  $\parallel v_n\parallel_{\mathbb{H}^2}=1$. 
Let $\sigma_n$ be isometries of $\mathbb{H}^2$ that send $(x_n,v_n)$ to a given pair $(x,v)$, and let $u_n:=u\circ \sigma_n$. 
As $u$ is Fuchsian, there exists $\beta <\pi/2$ such that $u\leq \beta$, and in turn 
$u_n\leq \beta$. By Lemma~\ref{lem:subseq cv}, up to consider a subsequence,
$(u_n)_n$ converges to a C-convex function $u_0$.
As we supposed that $\parallel v_n\parallel_u\to 0$, then $S_{u_0}$ must have a light-like support plane, that contradicts Lemma~\ref{convex then spacelike}.
\end{proof}

Note that Lemma~\ref{minor} indicates that in the Fuchsian case, $d_u$ is a distance and not only a pseudo-distance.

Let us recall the following classical result, see e.g. Lemma 3.14 in \cite{FFDS}. The homeomorphisms
in the statement below could also be constructed by hand, for example using canonical polygons as fundamental domains for the Fuchsian groups, see Section~6.7 in \cite{buser}.
 
 \begin{lemma} \label{tildephi}
 Let $(\Gamma_n)_n$ be a sequence of Fuchsian groups converging to a group $\Gamma$ and  $\tau_n$  the isomorphisms  given in Definition \ref{3.9}. There exist homeomorphisms $\phi_n:\mathbb{H}^2/\Gamma\longrightarrow\mathbb{H}^2/\Gamma_n$  whose lifts $\tilde{\phi}_n$ satisfy for any $\sigma\in\Gamma$, 
 \begin{equation*}
 \tilde{\phi}_n\circ\sigma=\tau_n(\sigma)\circ\tilde{\phi}_n   \label{**}
 \end{equation*}
 and such that  $(\tilde{\phi}_n)_n$ converges to the identity map uniformly on  compact sets i.e \begin{equation*} 
 \forall x\in\mathbb{H}^2, \tilde{\phi}_n(x)\underset{n\rightarrow\infty}{\longrightarrow}x  \label{*}
 \end{equation*} 
 \end{lemma}

%((CHANGER TT PAR $S_u,d_u...$ ))

 Now, let $u$ be a $C$-convex function and $S_u$ the surface described by $u$.   The length structure $\mathcal{L}_u$ given by $\eqref{lengthstructure}$ induces a (pseudo-)distance $d_{S_u}$.  In turn,   $d_{S_u}$  induces  a  length  structure  denoted  by $L_{d_{S_u}}$ defined in the following way:  the length of a curve $(c,u\circ c):[0,1]\rightarrow S_u$ is defined as %(for coherence we write the following sentence  in term of definition \ref{d_u})

\begin{align*}
L_{d_{S_u}}(c,u\circ c)&=\underset{\delta}{\sup}\sum\limits_{i=1}^{n}d_{S_u}((c(t_i),u\circ c(t_i)),(c(t_{i+1}), u\circ c(t_{i+1}))),\\
&=\underset{\delta}{\sup}\sum\limits_{i=1}^{n}d_u(c(t_i),c(t_{i+1}))=L_{d_u}(c) ~~~~~~(\text{see definition \ref{d_u} })
\end{align*}

%$$L_{d_{S_u}}(c,u\circ c)=L_{d_u}(c)=\underset{\delta}{\sup}\sum\limits_{i=1}^{n}d_{S_u}((c(t_i),u\circ c(t_i)),(c(t_{i+1}), u\circ c(t_{i+1}))),$$
 where the sup is taken over all the decompositions $$\delta=\{(t_1\ldots t_n)|t_1=0\leq t_2\leq...\leq t_n =1\},$$
    
    %A  curve  is rectifiable, if  its $L_{d_u}$-length  is  finite.   The  length  structure $L_{d_u}$ is lower-semicontinuous \cite[Proposition 2.3.4]{BurBur}: if a sequence of rectifiable curves $$c_n:[0,1]\rightarrow S,$$ converges to $c$ (i.e.  $c_n(t)\rightarrow c(t) $ for all $t$), then 
   % $$L_{d_u}(c)\leq \underset{n}{\liminf}L_{d_u}(c_n).$$

 %Note that,    in general we have the following inequality  $L_{d_u}\leq \mathcal{L}_u,$ (see \cite{Bur15} for more details and properties  of  the two length structures).
  We have the following proposition

 \begin{proposition}\label{converg}
    Let $(u_n)_n$ be a sequence of convex functions such that:\begin{itemize}
    \item $d_{u_n}$ is a complete distance  with Lipschitz shortest paths,
    \item $\mathcal{L}_{u_n}=L_{d_{u_n}}$ on the set of Lipschitz curves,
     \item There exists $0< R<\pi/2$ with $0\leq u_n<R $,
        \end{itemize}
     Then, up to extracting a subsequence, $(u_n)_n$ converges  to a convex function $u$ and $(d_{u_n})_n$ converges to $d_u$ uniformly on compact sets.\\

   \end{proposition} 
  \begin{proof}
  The proof of this proposition is similar as the one done in \cite[Proposition 3.12]{FFDS}. The proof was done using proposition \ref{convstruc}, the only difference is to use   Lemma \ref{major} and \ref{minor} instead of  \cite[corollary 3.11]{FFDS}.
  
  \end{proof}

 %Recall that for any convex  function $u_n$,  the  length structure $\mathcal{L}_{u_n}$ given by \eqref{lengthstructure} induces a pseudo-distance $d_{u_n}$, which itself gives a length structure $L_{d_{u_n}}$.   Now because we are using  approximation by smooth surfaces then  for all $n\in\mathbb{N}$ we consider  the functions   $u_n$ as  smooth functions, in this case  $d_{u_n}$ is a complete distance with Lipschitz  shortest paths. In particular,  we have  $\mathcal{L}_{u_n}=L_{d_{u_n}}$ (because of smoothness, see \cite{Bur15} for more details) and  this with     propositions \ref{converg} and \ref{5.10} give the  following.

We recall that in this paper  we are using approximation by smooth surfaces. 
 % in this case  $d_{u_n}$ is a complete distance with Lipschitz  shortest paths. 
We note also that by Lemma \ref{minor} and Lemma \ref{major},  $d_{u_n}$ are complete distances on $\mathbb{H}^2$, also   we have    $\mathcal{L}_{u_n}=L_{d_{u_n}}$ (because of smoothness, see \cite{Bur15} for more details),  %and  this with     propositions \ref{converg} and \ref{5.10} leads to the  following Lemma.
 we deduce the following

 %We deduce the following

 %\begin{lemma}\label{3.18}  
 %Let $(u_n,\Gamma_n)$ be  Fuchsian C-convex functions converging to a pair $(u,\Gamma)$.
 %Then on any compact set of $\mathbb{H}^2$, $d_{u_n}(\tilde{\phi}_n(.),\tilde{\phi}_n(.))$ uniformly converge to $d_u$, where $\tilde{\phi}$ is given by Lemma \ref{tildephi}.
 %\end{lemma}

 \begin{lemma}\label{3.18}  
Let $(u_n,\Gamma_n)$ be  Fuchsian C-convex functions such that:
\begin{itemize}
 \item $(u_n,\Gamma_n)_n$ converges to a pair $(u,\Gamma)$,
\item There exist $0<R<\pi/2$ with $0\leq u_n<R$,
 \item $d_{u_n}$ are distances with Lipschitz shortest paths,
 \item $d_{u_n}$ converge to $d_u$, uniformly on compact sets.
 \end{itemize}
 Then on any compact set of $\mathbb{H}^2$, $d_{u_n}(\tilde{\phi}_n(.),\tilde{\phi}_n(.))$ uniformly converge to $d_u$, where $\tilde{\phi}_n$ is given by Lemma \ref{tildephi}.
 \end{lemma}

\begin{proof}

By Lemma~\ref{minor} and Lemma~\ref{major}, the topology induced by $d_u$ onto $\mathbb{H}^2$ is the topology for the hyperbolic metric. It follows that for the maps $\tilde{\phi}_n$ of Lemma~\ref{tildephi}, we have that on compact sets, the maps $x\mapsto d_{u_n}(\tilde{\phi}_n(x),x)$ uniformly converge to $0$.  % Lemma~\ref{3.18}   then follows from classical manipulations and the convergence of length given by Lemma~\ref{convstruc}. See e.g. the proof of Corollary 3.18 in \cite{FFDS} for details.By Lemma \ref{major}, Lemma \ref{minor}  and Lemma \ref{tildephi}  we have that  $x\mapsto d_{u_n}(\tilde{\phi}_n(x),x)$ uniformly converges to $0$ on the given compact set).  
By the triangle inequality we have,

$$ d_{u_n}(\tilde{\phi}_n(x),\tilde{\phi}_n(y))-d_u(x,y)\leq d_{u_n}(\tilde{\phi}_n(x),x)+d_{u_n}(\tilde{\phi}_n(y),y)+d_{u_n}(x,y)-d_u(x,y)$$

   by the preceding arguments and proposition \ref{converg},  for $n$ sufficiently large  the right-hand side is uniformly less than any $\epsilon>0$. On the other  hand, by triangle inequality again we have,

%$d_u(x,y)-d_{u_n}(\tilde{\phi}_n(x),\tilde{\phi}_n(y))=d_u(x,y)-d_{u_n}(x,y)+d_{u_n}(x,y)-d_{u_n}(\tilde{\phi}_n(x),\tilde{\phi}_n(y))$

%\begin{eqnarray}
  % d_u(x,y)-d_{u_n}(\tilde{\phi}_n(x),\tilde{\phi}_n(y))& = &d_u(x,y)-d_{u_n}(x,y)+d_{u_n}(x,y)-d_{u_n}(\tilde{\phi}_n(x),\tilde{\phi}_n(y))    \\
  %        & \leq & d_u(x,y)-d_{u_n}(x,y)+ d_{u_n}(x,\tilde{\phi}_n(x))+d_{u_n}(y,\tilde{\phi}_n(y))\\&+&d_u(\tilde{\phi}_n(x),\tilde{\phi}_n(y)) -d_{u_n}(\tilde{\phi}_n(x),\tilde{\phi}_n(y))
%\end{eqnarray}

 \begin{align*}
   d_u(x,y)-d_{u_n}(\tilde{\phi}_n(x),\tilde{\phi}_n(y))& = &d_u(x,y)-d_{u_n}(x,y)+d_{u_n}(x,y)-d_{u_n}(\tilde{\phi}_n(x),\tilde{\phi}_n(y))    \\
\         & \leq & d_u(x,y)-d_{u_n}(x,y)+ d_{u_n}(x,\tilde{\phi}_n(x))+d_{u_n}(y,\tilde{\phi}_n(y))\\
&&+~~~~d_{u_n}(\tilde{\phi}_n(x),\tilde{\phi}_n(y)) -d_{u_n}(\tilde{\phi}_n(x),\tilde{\phi}_n(y))
 \end{align*}

which  is uniformly less than any $\epsilon>0$ for $n$ sufficiently large (by the same arguments). % and the term 

%is  less than  $d_{u_n}(x,\tilde{\phi}_n(x))+d_{u_n}(y,\tilde{\phi}_n(y))-d_{u_n}(\tilde{\phi}_n(x),\tilde{\phi}_n(y))$ that is uniformly less than any $\epsilon>0$ for $n$ sufficiently large.
\end{proof}

$ $ \\

By definition, if $(u,\Gamma)$ is a Fuchsian C-convex function, then $\Gamma$ acts by isometries on $d_u$. In turn, $d_u$ defines a distance on the  compact surface $\mathbb{H}^2/\Gamma$.

\begin{definition}\label{def bard}
For a Fuchsian C-convex function $(u,\Gamma)$, we denote by $\bar{d}_u$ the distance defined by $d_u$ on $\mathbb{H}^2/\Gamma$.
\end{definition}

The reason to introduce the maps $ \tilde{\phi}_n$ from Lemma~\ref{tildephi} is the following Corollary of Lemma~\ref{3.18}. Its proof is formally the same as the one of Proposition~3.19 in \cite{FFDS}. (The definition of uniform convergence of metric spaces is recalled in Definition~\ref{def unif conv}.)

\begin{proposition}\label{5.10}
 Let $(u_n,\Gamma_n)$ be  Fuchsian C-convex functions converging to a pair $(u,\Gamma)$. Up to extracting a subsequence, $(\mathbb{H}^2/\Gamma_n,\bar{d}_{u_n})_n$ uniformly converges to $(\mathbb{H}^2/\Gamma,\bar{d}_{u})$.
 \end{proposition}

\subsection{Convergence of metrics implies convergence of groups}

The aim of this section is to prove Proposition~\ref{this propo}, 
that may be seen as a kind of converse of Proposition \ref{5.10}. 
The distance $\bar{d}_{u}$ was defined in Definition~\ref{def bard}.

 \begin{proposition}\label{this propo}
 Let $(S,d)$ be a metric of curvature $\leq -1$ and let $(u_n,\Gamma_n)$ be smooth Fuchsian C-convex functions, such that the sequence  $(\mathbb{H}^2/\Gamma_n,\bar{d}_{u_n})_n$ uniformly converges to $(S,d)$. Up to extracting a subsequence,
  \begin{itemize}
  \item $(\Gamma_n)_n$ converges to a Fuchsian group $\Gamma$;
 \item there exists $0<\beta<\pi/2 $ such that  $0\leq u_n<\beta$.
  \end{itemize}
  \end{proposition}

Under the hypothesis of Proposition \ref{this propo}, let's first   prove the convergence of groups. We first have a consequence of simple hyperbolic geometry, see \cite[Corollary 4.2]{FFDS}.

\begin{lemma} \label{geq}
There exists $G>0$ and $N>0$ such that for any $n>N$, for any $x\in\mathbb{H}^2$,  for every element $\sigma_n\in\Gamma_n\setminus\{0\}$
$$d_{u_n}(x,\sigma_n(x))\geq G.$$ 
\end{lemma}

\begin{proposition} 
Under the hypothesis of Proposition \ref{this propo}, up to extracting a subsequence, the sequence $(\Gamma_n)_n$ converges to a Fuchsian group $\Gamma$.
\end{proposition}
\begin{proof}
First  by Lemma \ref{major} we have that for all $x,y\in\mathbb{H}^2$,  $$d_{u_n}(x,y)\leq d_{\mathbb{H}^2}(x,y),$$ and by Lemma \ref{geq}, we have that  there exists $G>0$ and $N>0$ such that for any $n>N$ and for any $x\in\mathbb{H}^2$:
$$G\leq d_{u_n}(x,\sigma_n(x))\leq d_{\mathbb{H}^2}(x,\sigma_n(x)), $$ 
in particular  if $$L_{\sigma_n}=\min_{x\in\mathbb{H}^2}d_{\mathbb{H}^2}(x,\sigma_n(x)),$$  we have $$G\leq L_{\sigma_n}.$$ 

The length is uniformly bounded from below, hence by a classical result of Mumford \cite{Mumf} we can deduce that up to extracting a subsequence, the sequence of groups converges. 
\end{proof}

 \begin{lemma}\label{not all the surface} 
 Under  the assumptions of Proposition~\ref{this propo}, there exists $M<\pi/2$ such that
 for all $n$, there is $x_n\in \mathbb{H}^2$ such that $u_n(x_n) < M$.
 \end{lemma} 
\begin{proof}
Suppose that the result is false: for a sequence $M_k\to \pi/2$, there is $n_k$ such that $u_{n_k} \geq M_k$. By the definition of the length structure \eqref{lengthstructure}, it follows that 
$d_{u_{n_k}}\leq \cos M_k d_{\mathbb{H}^2}$. In turn, $(\mathbb{H}^2/\Gamma_n,\bar{d}_{u_n})_n$ has a subsequence converging to $0$, that is a contradiction.
\end{proof}

\begin{proposition}\label{majorfunc}
Under the hypothesis of Proposition~\ref{this propo}, there exists $0<\beta<\pi/2$ such that,  for any $n\in\mathbb{N}$, for any $x\in\mathbb{H}^2$,
$$u_n(x)<\beta.$$ 
\end{proposition}

\begin{proof}
Let us consider the affine model of anti-de Sitter space. 
As the sequence of groups converges, there exists a compact set $C\subset \mathbb{D}$, which contains a fundamental domain for $\Gamma_n$ for all $n$. Hence the points $x_n$ given by  Lemma~\ref{not all the surface} can be chosen to all belong to $C$.
The result follows because the convex maps $\bar{u}_n$ on the disc are zero on the boundary, so for any compact set $C$ in the interior of the disc, the difference between the minimum and the maximum of  $\bar{u}_n$ on $C$ cannot be arbitrary large.
\end{proof}

Proposition~\ref{this propo} is now proved.

 \section{Proof of Theorem \ref{1.4}}\label{sec proof}
 
 The proof relies on the two following results. 
 
\begin{theorem}[{\cite{LF00}}]\label{ScH}
 Let $(S,d)$ be a  metric induced by a Riemannian metric of sectional curvature $<-1$. Then there exists 
a $C^\infty$ Fuchsian C-convex   $u:\mathbb{H}^2\to [0,\pi/2[$ such that $\bar{d}_u$ is
isometric to $d$.
\end{theorem}

\begin{theorem}\label{thm:approx smooth}
Let $(S,d)$ be a metric of curvature $\leq-1$. Then  there exists a sequence $(S_n,d_n)$ converging uniformly to $(S,d)$, where $S_n$ are homeomorphic to $S$ and $d_n$ are induced by Riemannian metrics  with  sectional curvature $<-1$  .
\end{theorem} 

Although Theorem~\ref{thm:approx smooth} may seem well-known, we didn't find any reference for it, so we will prove it in Section~\ref{sec:approx smooth}.
Note that we are not aware if the analogue of Theorem~\ref{thm:approx smooth} holds for metrics of curvature $\leq 1$.

Let $d$ be a metric of  curvature $\leq -1$ on  $S$. From Theorem~\ref{thm:approx smooth},
there exists a sequence $(d_n)_n$ of  metrics induced by Riemannian metrics with sectional curvature $< -1$ on $S$ that   converges uniformly to $d$. 
 By  Theorem~\ref{ScH}, for each $n\in\mathbb{N}$ there exists a Fuchsian C-convex pair $(u_n,\Gamma_n)$ such that $\bar{d}_{u_n}$ is isometric to $d_n$ and $u_n$ is smooth. 
By Proposition~\ref{this propo} there is a subsequence of $(\Gamma_n)_n$ converging to a Fuchsian group $\Gamma$, and $\beta<\pi/2$ such that $0\leq u_n<\beta$. 
 
 So Lemma~\ref{lem:subseq cv} and Proposition~\ref{5.10} applies: up to extracting a subsequence, there is a function $u$ such that the induced distance on  $\bar{d}_u$ (the quotient of $d_u$ by $\Gamma$) is the uniform limit of $(\mathbb{H}^2/\Gamma_n,\bar{d}_{u_n})$, i.e. the uniform limit of $(S,d_n)$.
The limit for uniform convergence is unique,  up to isometries \cite{BurBur}, 
so $\bar{d}_u$  is isometric to $d$.  Theorem~\ref{1.4} is proved, with $L$ the quotient of the AdS cylinder of Section~\ref{sec ads} by $\Gamma$.

\section{Approximation by smooth metrics}\label{sec:approx smooth}

 In the following we will use the uniform convergence, so let's recall its definition.
  \begin{definition} \label{def unif conv}
    We say that a sequence of metric spaces $(S_n,d_n)_n$  \emph{converges uniformly} to the metric space $(S,d)$ if there exist homeomorphisms  $f_n:S\longrightarrow S_n$ such that 
    $$\underset{x,y\in S}{\sup}|d_n(f_n(x),f_n(y))-d(x,y)|\xrightarrow[n\to\infty]{} 0~.$$
            \end{definition}
If $S_n=S$ and $f_n=id$, then this is the usual definition of uniform convergence of distance functions.% for the functions $d_n$ defined on $S\times S$. 

We want to check that a metric of curvature $\leq -1$ on the closed surface $S$ can be approximated (in the sense of the uniform convergence) by distances induced by Riemannian metrics with sectional curvature $<-1$.
We will first  approximate by hyperbolic metrics with conical singularities 
of negative curvature. Then we will ``smooth'' those cone metrics.

\subsection{Approximation of metrics by polyhedral metrics}

 Let $(X,d_0)$ be a   metric space such that  every pair of points can be joined by a shortest path.
  A \emph{(geodesic) triangle} $\Delta$ of $X$ consists of three points $x,y,z\in X$ and  shortest paths $[x,y],[y,z]$ and $[z,x]$. A hyperbolic 
  comparison triangle for $\Delta$ is a geodesic triangle $\tilde{\Delta}$ in the hyperbolic
   space  with vertices $\tilde{x},\tilde{y}$ and $\tilde{z}$, such that $d_0(x,y)=d_{\mathbb{H}^2}(\tilde{x},\tilde{y})$    , $d_0(y,z)=d_{\mathbb{H}^2}(\tilde{y},\tilde{z})$, $d_0(x,z)=d_{\mathbb{H}^2}(\tilde{x},\tilde{z})$. The interior angle of   $\tilde{\Delta}$ at  $\tilde{x}$ is called the  \emph{comparison angle} at $x$ of the triangle $\Delta$.

\begin{definition} Let $\gamma,\gamma'$ be two non trivial shortest paths issued from the same point $x$.
Let $\tilde{\angle}(\gamma(t)x\gamma'(t'))$ be the angle at $\tilde{x}$ of the  comparison  triangle $\tilde{\Delta}$ with vertices $\tilde{\gamma}(t),\tilde{x}$ and $\tilde{\gamma}'(t')$ in the hyperbolic  plane  corresponding to the triangle $\Delta(\gamma(t)x\gamma(t'))$ in $X$. Then the \emph{upper angle}  at $x$ of $\gamma$ and $\gamma'$  is defined by 
\begin{equation}\label{eq:angle sup}\underset{t,t'\longrightarrow 0}{\limsup \tilde{\angle}}(\gamma(t)x\gamma'(t')).\end{equation}
\end{definition}

\begin{definition}
We say that an intrinsic metric space $(X,d_0)$ is  CAT$(-1)$ if the upper angle between any couple of sides of every geodesic triangle with distinct vertices is not greater than the angle between the corresponding sides of its
comparison triangle in the hyperbolic plane. 
\end{definition}

Let $B_{d_0}(x,r)$ be the ball of center  $x$ and radius $r$ in $(X,d_0)$.
\begin{definition} \label{curvature}
An intrinsic metric space  $(X,d_0)$ has  curvature $\leq -1$ (in the Alexandrov sense), if for any  $x$ there exists  $r$ such that  $B_{d_0}(x,r)$ endowed with the induced (intrinsic) distance is CAT$(-1)$.
\end{definition}

Let us remind the notion of bounded integral curvature \cite[Chapter I, p. 6]{Alek67}. 
A \emph{simple triangle} is a triangle bounding an open set homeomorphic to a disc, consisting of three distinct points (the vertices of the triangle) and three shortest paths joining these points,  and which is convex relative to the boundary, i.e.  no two points of the boundary of the triangle, can be joined by a curve outside the triangle, which is shorter than a suitable part of the boundary joining the points, (see \cite{Alek67} for more details).

\begin{definition}
An intrinsic distance $d_0$ on a surface $S$ is said to be of bounded integral curvature (in short, BIC), if $(S,d_0)$ verifies the following property: 

 For every $x\in S$ and every neighborhood $N_x$ of $x$ homeomorphic  to an open disc, for any finite system $\mathcal{T}$ of pairwise non-overlapping simple triangles $T$ belonging to $N_x$, the sum of the excesses 
$$\delta_0(T)=\bar{\alpha}_T+\bar{\beta}_T+\bar{\gamma}_T-\pi,$$
of the triangles $T\in\mathcal{T}$ with upper angles $(\bar{\alpha}_T,\bar{\beta}_T,\bar{\gamma}_T)$ is bounded from
above by a number $C$ depending only on the neighborhood $N_x$, i.e. 

$$\sum\limits_{T\in\mathcal{T}}\delta_0(T)\leq C.$$

\end{definition}

The main tool for our approximation result is the following.  

\begin{theorem}[{\cite[Theorem 2 p. 59]{Alek67}}] \label{BIC}
Let  $\epsilon>0$. A compact $BIC$ surface admits a triangulation by a  finite number of arbitrary   non overlapping simple triangles of diameter $<\epsilon$.
\end{theorem}
We will also need the following result to prove that the sum of the angles in a cone point is not less than $2\pi$, it corresponds to Theorem 11 in \cite[Chapter II, p. 47]{Alek67}.

\begin{lemma}\label{pBIC}
Let $p$ be a point on a $BIC$ surface such  that  there  is  at  least  one shortest  arc  containing $p$ in  its  interior.  Then  for  any  decomposition  of  a  neighborhood of $p$ into sector convex relative to the boundary formed by geodesic rays issued from $p$ such that the upper angles between the sides of these sectors exist and
do not exceed $\pi$, the total sum of those angles is not less than $2\pi$.
\end{lemma}

To get a triangulation of our surface, we will  use some properties of BIC surfaces, so let's consider the following  Lemma.
\begin{lemma} \label{CATT}
A metric of curvature $\leq -1$  is a  $BIC$ surface.
\end{lemma}
\begin{proof}
We have that a CAT$(0)$ surface is a  BIC surface  (proved in \cite{FFDS}, Lemma 2.18), and that a CAT$(-1)$ surface is also a CAT$(0)$ surface (see \cite{BH99}, Chapter II, page 165). The Lemma follows because of the local nature of the definition of BIC and curvature $\leq -1$.
\end{proof}

For a BIC surface, the angle exists, that means that in \eqref{eq:angle sup}
the limit exists in place of the limsup \cite{Alek67}. So in the following we will speak about angles rather than upper angles.

\begin{theorem} \label{approx}
Let $(S,d)$ be a metric of curvature $\leq-1$ on the closed surface $S$. 
Then there exists a sequence $(S_n,d_n)$  converging uniformly to $(S,d)$, where $S_n$ is homeomorphic to $S$ and $d_n$ is the metric induced by  a  
hyperbolic metric with conical  singularities of negative curvature on $S_n$.
\end{theorem} 

The remainder of this section is devoted to the proof of Theorem~\ref{approx}.

Applying  Theorem \ref{BIC} and Lemma \ref{CATT}, we obtain a triangulation $\mathcal{T}_\epsilon$ of our surface in which every simple triangle has  diameter $<\epsilon$. Replace the interiors of the triangles of $\mathcal{T}_\epsilon$ by the interiors of the hyperbolic comparison triangles.  We obtain  $(\bar{S}_\epsilon,\bar{d}_\epsilon)$ which is a hyperbolic metric with conical singularities, corresponding to the vertices of the triangles.  By construction, $\bar{S}_\epsilon$ is endowed with a triangulation $\bar{\mathcal{T}}_\epsilon$. 

\begin{lemma}
The total angles around the conical singularities of $\bar{d}_\epsilon$  are not less than $2\pi$.
\end{lemma}
\begin{proof}
 By a property  of the  CAT$(-1)$ spaces, we have that every vertex of $\mathcal{T}_\epsilon$ lies in the interior of some geodesic in $(S,d)$ \cite[II.5.12]{BH99}. Applying Lemma \ref{pBIC} we immediately get that the sum of the sector angles $\alpha_i$ at any vertex $V$ of the triangulation $\mathcal{T}_\epsilon$ in $(S,d)$ is not less than $2\pi$. By definition of the CAT$(-1)$ spaces, we have that  the angles $\alpha_{-1,i}$ of the comparison triangles in the hyperbolic space are not less than the corresponding angles  at every vertex $V$ in the triangulation $\mathcal{T}_\epsilon$ in $(S,d)$.  It follows that $$2\pi\leq\sum_i\alpha_i\leq\sum_i\alpha_{-1,i}~.$$ 
\end{proof}

We want to prove that the finer the triangulation is, the closer  $\bar{d}_\epsilon$ is from $d$ (for the uniform convergence between metric spaces). This relies on a series of lemmas.

 \begin{lemma}\label{samepoint}
 Let $\alpha$ be the angle at a vertex of a  triangle  $T$ in a  surface of curvature $\leq -1$, and let $\alpha_{-1}$ be the corresponding angle in a comparison triangle $T_{-1}$ in the hyperbolic space then,
  $$\alpha_{-1}-\alpha\leq -area(T_{-1})-\delta_0(T).$$
 \end{lemma}
 \begin{proof}
 If $\beta$ and $\lambda$ are the other angles of $T$ and $\beta_{-1}$, $\lambda_{-1}$ the corresponding angles in $T_{-1}$, then we have   $$\alpha_{-1}-\alpha\leq \alpha_{-1}-\alpha+\beta_{-1}-\beta+\lambda_{-1}-\lambda=\delta_0(T_{-1})-\delta_0(T)=-area(T_{-1})-\delta_0(T).$$
 \end{proof}
 \begin{lemma}\label{sumsigma}
 If $\mathcal{T}$ is a triangulation of  a compact surface $(S,d)$ with curvature $\leq -1$  by non overlapping simple triangles, then 
 $$\sum\limits_{T\in\mathcal{T}}\delta_0(T)\geq 2\pi\mathcal{X}(S),$$
 with $\mathcal{X}(S)$ the Euler characteristic of $S$. 
 \end{lemma}
 \begin{proof}
 Let $|T|$ be the number of triangles, $|E|$ the number of edges  and $|N|$ the number of vertices  in our geodesic triangulation. We have $|E|=\frac{3}{2}|T|$, so that the Euler formula   $$|T|-|E|+|N|=\mathcal{X}(S)$$ implies 
 \begin{equation}
 2|N|-|T|=2\mathcal{X}(S).\label{vertex}
  \end{equation}
 If we denote by $\theta_i$ the sum of the angles of the triangles around a vertex, then using \eqref{vertex} it follows that 
 
 $$\sum\limits_{T\in\mathcal{T}}\delta_0(T)=\sum\limits_{i=1}^{N}\theta_i-|T|\pi=\sum\limits_{i=1}^{N}(\theta_i-2\pi)+2\pi\mathcal{X}(S)$$
 
 The proof follows because $\theta_i-2\pi\geq 0$ for all $i$.
 \end{proof}

\begin{lemma}\label{fkepsi}
Let $T$ be an isosceles triangle in the hyperbolic space with diameter less than a given $ \epsilon$ and with edges length $x,x,l$ and $\theta$ the angle opposite to the edge of length $l$, then 
$$l\leq \sinh(\epsilon)\theta.$$ 
\end{lemma}
\begin{proof}
By the hyperbolic cosine law, 
$$\cosh(l)=\cosh^2(x)-\sinh^2(x)\cos(\theta),$$
that is equivalent to 
$$1+2\sinh^2(\frac{l}{2})=\cosh^2(x)-\sinh^2(x)(1-2\sin^2(\frac{\theta}{2})),$$

so 

$$\frac{l}{2}\leq\sinh(\frac{l}{2})=\sinh(x)\sin(\frac{\theta}{2})\leq\sinh(\epsilon)\frac{\theta}{2}.$$

\end{proof}

\begin{lemma}\label{OAB} Let $\epsilon>0$.
Let $T$ be a simple triangle in $(S,d)$ of diameter  $<\epsilon$ with vertices $OXY$. Let $A$ (resp. $B$) be on the edge $OX$ (resp. $OY$) and at distance $a$ (resp. $b$) from $O$. Let $T_{-1}^1$ be a comparison triangle  for $T$ in the hyperbolic space, with vertices $O'X'Y'$. Let $A'$ (resp. $B'$) be the corresponding point of $A$ (resp. $B$) (i.e on  the edge $O'X'$  (resp. $O'Y'$) and at distance $a$ (resp. $b$) from $O'$), then 
$$0\leq d_{\mathbb{H}^2}(A',B')-d(A,B)\leq -\delta_0(T)\sinh(\epsilon).$$ 
\end{lemma}

\begin{proof} 
The first inequality comes from  the fact that we are in a CAT$(-1)$ neighborhood (\cite{BH99}, page 158). Let $T^2_{-1}$ be the comparison triangle for $OAB$ in the hyperbolic space drawn such that the edge of length $a$ is identified with $O'A'$ (see Figure~\ref{Trianglo}). Let $B''$ be the  corresponding comparison point  for $B$ in $T^2_{-1}$ (i.e. $B''$ satisfies  $d(A,B)=d_{\mathbb{H}^2}(A',B'')$ and $d(O,B)=d_{\mathbb{H}^2}(O',B'')$). By the triangle inequality we have 
\begin{equation}
d_{\mathbb{H}^2}(A',B')-d(A,B)=d_{\mathbb{H}^2}(A',B')-d_{\mathbb{H}^2}(A',B'')\leq d_{\mathbb{H}^2}(B',B'').\label{ineq}
\end{equation}

Let $\theta_1$ be the angle at $O'$ of $T^1_{-1}$ (i.e. the angle at $O'$ of $O'A'B'$), and let $\theta_2$ be the angle at $O'$ of $T^2_{-1}$ (i.e. $O'A'B''$).

 % By the hyperbolic cosine  law, $\theta_1-\theta_2\neq 0$ moreover,

  We have  $\theta_1-\theta_2$ is the angle at $O'$ of $O'B'B''$ which is isosceles  so by inequality \eqref{ineq} and  Lemma \ref{fkepsi} it follows that 
$$d_{\mathbb{H}^2}(A',B')-d(A,B)\leq \sinh(\epsilon)(\theta_1-\theta_2).$$ 
 If $\beta$ is the angle of $T$ at $O$, then both $\theta_1$ and $\theta_2$ are angles corresponding to $\beta$ in the different comparison triangles, so by Lemma \ref{samepoint} 
 $$\theta_1-\theta_2=\theta_1-\beta+\beta-\theta_2\leq \theta_1-\beta\leq -area(T_{-1}^1)-\delta_0(T)$$
 that leads to the result.
 \end{proof}
 \begin{center}
  \begin{figure}[H] \begin{center}
\includegraphics[scale=0.9]{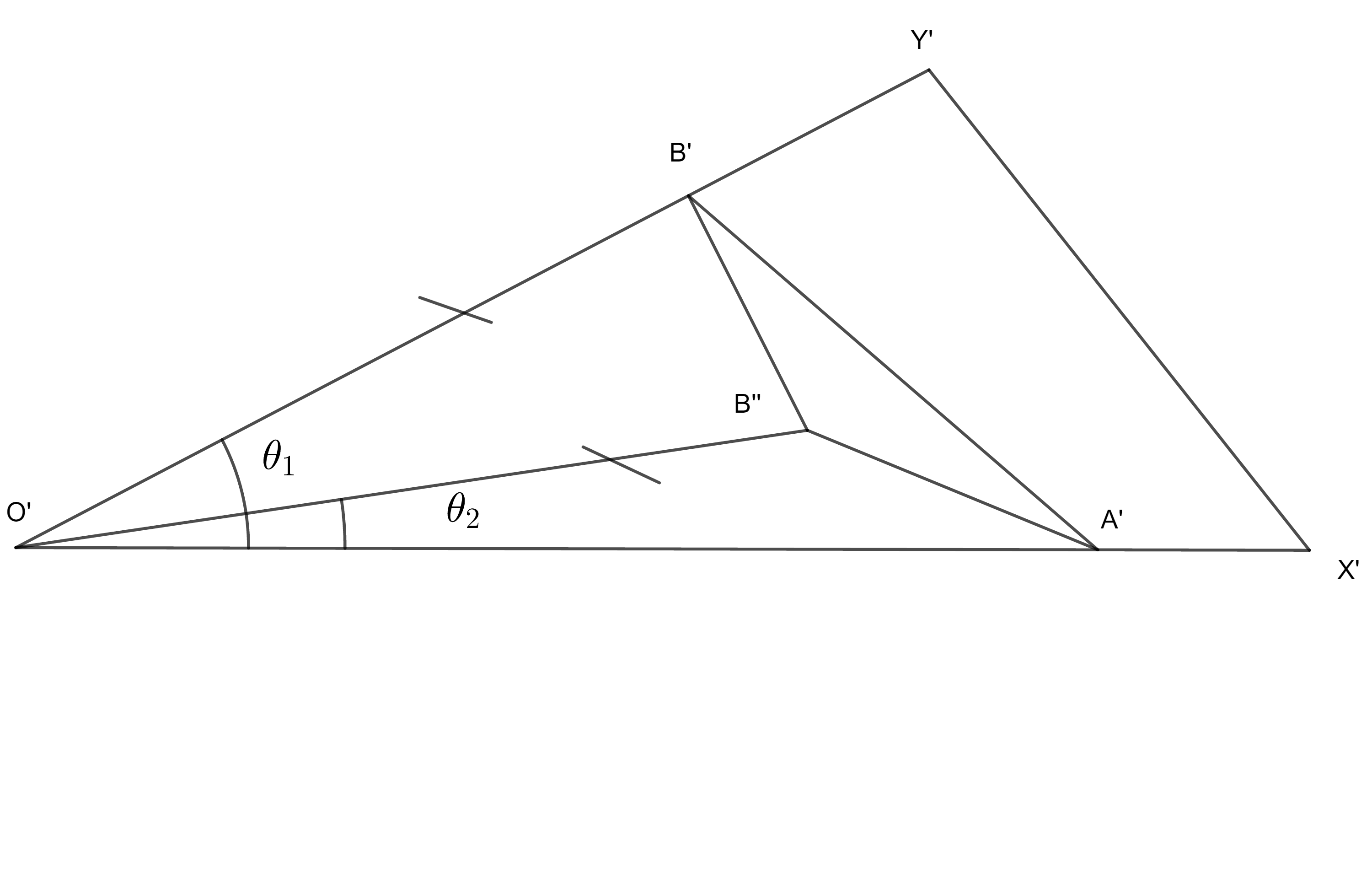}
 \caption{Notations for the proof of Lemma \ref{OAB}.}\label{Trianglo}\end{center}
  \end{figure}
\end{center}

Now, let's describe a homeomorphism between $(S,d)$ and $(\bar{S}_\epsilon,\bar{d}_\epsilon)$ in the following way. 
The triangle $\bar{T}_i$  do not degenerate into segments, since the sum of every two sides is greater than the third. Therefore, the triangles $T_i$ can be mapped homeomorphically onto the corresponding triangles $\bar{T}_i$, such that the vertices are sent to vertices and the homeomorphism restricts to an isometry along the edges.  We consider any homeomorphism from the interior of the triangles that extend the homeomorphism on the boundary.
As the surfaces are triangulated  by such triangles, this gives a homeomorphism from $S$ to $\bar{S}_\epsilon$.

For two points $H$ and $J$ on $S$, we denote by $H'$, $J'$ the correponding points on  $\bar{S}_\epsilon$.

\begin{fact}\label{4.11} With the notations above,
$-2 \epsilon \leq \bar{d}_\epsilon(H',J')-d(H,J)\leq 2\epsilon-2\pi\mathcal{X}(S)\sinh(\epsilon)$. 
\end{fact}
\begin{proof}
 The idea of this proof  is the same as  \cite[Lemma 2, page 263]{AL06}.  Let's prove the first inequality. Let  $H',J'\in  \bar{S}_\epsilon $ and $\gamma'$ a shortest path  joining $H'$ and $J'$   and  $\gamma$ be a path joining $H$ and $J$  such that the intersection with every triangle $T$ is a shortest path  (i.e. each connected piece of  $\gamma'$ meeting a triangle $T'$ from a point $A'$ to a point $B'$ on the boundary of $T'$ is associated in $T$ the shortest path joining the corresponding (in the sense of Lemma \ref{OAB}) points  $A$  and $B$).

      Let us denote by $\gamma'_i,~i=0,\ldots, m+1$ the decomposition of $\gamma'$ given by the triangles  it crosses, and by $l(\gamma'_i)$ their lengths.

 As $(S,d)$ is CAT$(-1)$, the length of a connected component of the intersection of $\gamma'$ with $T'$ joining two points of the boundary is greater than the length of the corresponding component of $\gamma$ in $T$ (\cite{BH99},  page 158). Now, because the diameters are not greater than $\epsilon$ then   $l(\gamma_0)+l(\gamma_{m+1})\leq 2\epsilon$  and $l(\gamma'_0)+l(\gamma'_{m+1})\leq 2\epsilon$.  It follows that 
\begin{equation*}
d(H,J)\leq \sum\limits_{i=1}^{m}l(\gamma_i)+2 \epsilon\leq\sum\limits_{i=1}^{m}l(\gamma'_i)+2 \epsilon\leq  \bar{d}_\epsilon(H',J')+ 2\epsilon \end{equation*}
 then $$-2 \epsilon \leq \bar{d}_\epsilon(H',J')-d(H,J).$$

  The first inequality is now proved. 

Let's  now prove  the second  inequality. For that, consider  a shortest path $\gamma$ joining $H$ and $J$ in $S$ and  $\gamma'$ be a path in $\bar{S}_\epsilon$ joining $H'$ and $J'$  such that the intersection with every triangle $T'$ is a shortest path  (i.e. each connected piece of  $\gamma$ meeting a triangle $T$ from a point $A$ to a point $B$ on the boundary of $T$ is associated in $T'$ the shortest path joining the corresponding (in the sense of Lemma \ref{OAB}) points  $A'$  and $B'$).

%$ $ \\

 %obtained in the following way:  each connected piece of  $\gamma$ meeting a triangle $T$ from a point $A$ to a point $B$ on the boundary of $T$ is associated in $T'$ the shortest path joining the corresponding points $A'$ and $B'$.

  Let us denote by $\gamma_i,~i=0,\ldots, m+1$ the decomposition of $ \gamma$ given by the triangles  it crosses, and by $l(\gamma_i)$ their lengths, we find

 $$\bar{d}_\epsilon(H',J')-d(H,J)\leq l(\gamma_0')+l(\gamma_{m+1}')+\sum\limits_{i=1}^{m}l(\gamma_i')-l(\gamma_i).$$

Since $l(\gamma_0')$ and $l(\gamma_{m+1}')$ are not greater  than $\epsilon$   then    $ l(\gamma_0')+l(\gamma'_{m+1})\leq2\epsilon$.    By Lemma \ref{OAB} it follows that  $$\bar{d}_\epsilon(H',J')-d(H,J)\leq 2\epsilon-\sum\limits_{i=1}^{m}\delta_0(T_i)\sinh(\epsilon),$$

 But $\delta_0(T_i)$ are non positive and moreover the triangles $T$ are relative  convex, so $\gamma$ meets each triangle at most once (because, if the shortest path  $\gamma$ meets the (geodesic) triangle more than once, then there will be two points on the boundary of the triangle joined by a shortest path lying outside of the triangle, that contradicts the fact that the triangles are convex relative to the boundary), so $-\sum\limits_{i=1}^{m}\delta_0(T_i)$ is less than $-\sum_T\delta_0(T)$ for all the triangles of the triangulation of $S$, which is less than $-2\mathcal{X}(S)$ by Lemma \ref{sumsigma}. The second inequality is now proved. This fact is now proved.\\
\end{proof}

The lemmas above imply the uniform convergence. Theorem  \ref{approx} is now proved. 
 
\subsection{Approximation of polyhedral metrics by smooth metrics  }

\begin{proposition}\label{smoothappr}
Let $d$ be the metric induced by  a  
hyperbolic metric with conical  singularities of negative curvature on the closed surface $S$.
Then there exists a sequence $(S_n,d_n)$  converging uniformly to $(S,d)$, where $S_n$ is homeomorphic to $S$, $d_n$ is metric induced by a Riemannian metric of  sectional curvature $<-1$.
\end{proposition}
We use the same method  as that in \cite[Lemma 3.9]{DSth}, but we choose the cone in  anti-de Sitter space (Figure \ref{AdSCone2}), rather than the hyperbolic space $\mathbb{H}^3$.
\begin{proof}
Let $p\in S$ be a singular point of the polyhedral hyperbolic metric $d$. Consider a neighborhood  $U_p$ of $p$ in $S$ which doesn't contain any other singular point of $d$. As the curvature is supposed to be negative, the neighborhood $U_p$ equipped with the restriction of the metric $d$ will be isometric to the neighborhood  of a space-like circular cone $C_p$ in the affine model of the anti-de Sitter space,  such that the singularity $p$ corresponds to  the apex of $C_p$. 
Consider a sequence of smooth convex functions,  whose graphs coincide with the cone $C_p$  outside a neighborhood of the apex, and converging to $C_p$ (this is very classical, see e.g. 
\cite[Lemma 3.9]{DSth}).

Using Gauss formula, one can easily check that the sectional curvature for the induced metric on the smooth approximating surfaces is  $\leq -1$.  We  can multiply those metrics by any constant  $\lambda>1$ to get the sectional  curvature $<-1$. As the surfaces differ only on a compact set, and as the approximating sequence 
is smooth, it follows from \eqref{lengthstructure} that the induced distances are uniformly bi-Lipschitz to the hyperbolic metric. From this and Proposition~\ref{convstruc}, it is classical to deduce that  the induced distances converges locally uniformly (hence uniformly in this case), see e.g. the proof of Proposition 3.12 in \cite{FFDS}. 

The proposition follows by applying this procedure simultaneously  to all singular points of the metric $d$. 

 \end{proof}
 
\begin{center}

\begin{figure}[H]
\includegraphics[scale=0.15]{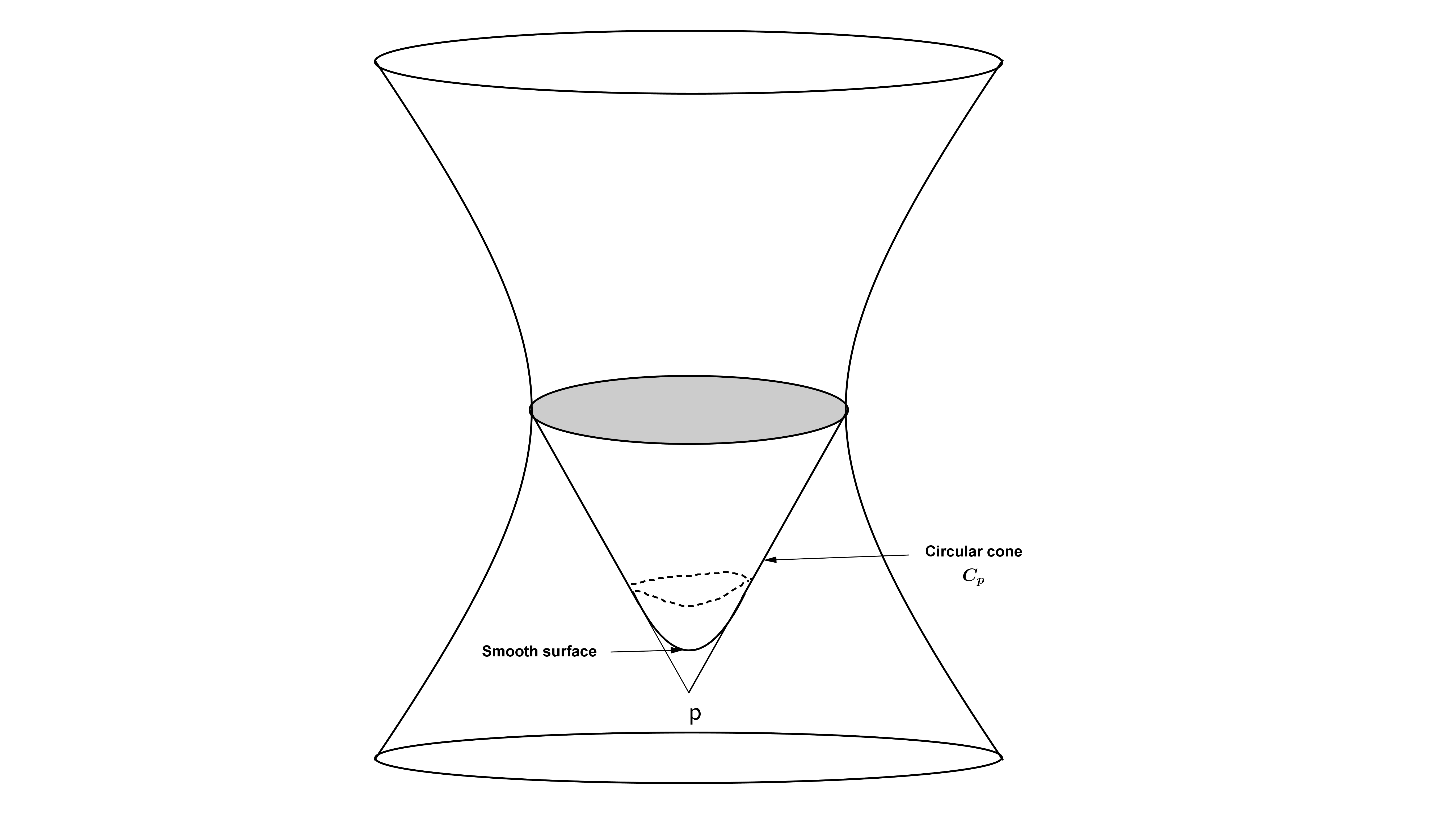}
\caption{Smooth surface and circular cone in Anti-de Sitter space.}\label{AdSCone2}
\end{figure}
 
 \end{center}

Let  $d$ be any metric of curvature $\leq -1$ on a compact surface $S$. 
We obtain a sequence $(d_n)_n$ from  Theorem~\ref{approx}, and for each $d_n$, 
a sequence $(d_{n_k})_k$ from   Proposition~\ref{smoothappr}.  Theorem~\ref{thm:approx smooth} follows from a diagonal argument.
 
\begin{footnotesize}
\nocite{}
\bibliography{arxiv}
\bibliographystyle{alpha}
\end{footnotesize}

 \vspace{1cm}
\noindent
Hicham Labeni \\
CY Cergy Paris Universit\'e \\
Laboratoire AGM \\
UMR 8088 du CNRS \\
F-95000 Cergy \\
France
\\ hicham.labeni@u-cergy.fr

\end{document}